\documentclass[reqno,11pt]{amsart}
\usepackage{amsmath}
\usepackage{amssymb}
\usepackage{tabularx}
\usepackage{enumerate}
\usepackage[dvips]{psfrag,graphicx}
\usepackage{color}
\usepackage{subfigure}
\usepackage{cases}
\usepackage{bbm}  

\usepackage{mathrsfs}
\usepackage{empheq}
\usepackage{amsfonts,amssymb,amsmath}
\usepackage{epsfig}

\usepackage{comment}
 
\makeatletter
\@addtoreset{equation}{section}
  
\makeatother
\newtheorem{thm}{Theorem}[section]
\newtheorem{lem}[thm]{Lemma}

\theoremstyle{definition}

\theoremstyle{definition}

\begin{document}
\title[Bifurcation structure of 
steady-states for a cooperative model]
{Bifurcation structure of steady-states for a cooperative
model with population flux by attractive transition}
 \thanks{This research was
partially supported by JSPS KAKENHI Grant Number 22K03379.}
\author[M. Adachi]{Masahiro Adachi$^\dag$}
\author[K. Kuto]{Kousuke Kuto$^\ddag$}
\thanks{$\dag$ Department of Pure and Applied Mathematics, 
Graduate School of Fundamental Science and Engineering,
Waseda University, 
3-4-1 Ohkubo, Shinjuku-ku, Tokyo 169-8555, Japan.}
\thanks{$\ddag$ Department of Applied Mathematics, 
Waseda University, 
3-4-1 Ohkubo, Shinjuku-ku, Tokyo 169-8555, Japan.}
\thanks{{\bf E-mail:} \texttt{kuto@waseda.jp}}
\date{\today}

\begin{abstract} 
This paper studies the steady-states
to a diffusive Lotka-Volterra cooperative model
with population flux by attractive transition.
The first result gives many bifurcation points
on the branch of the positive constant solution
under the weak cooperative condition.
The second result shows each steady-state approaches
a solution of the scalar field equation as
the coefficients
of the flux tend to infinity.
Indeed, the numerical simulation using pde2path
exhibits 
the global bifurcation branch
of the cooperative model with large population flux
is near that of
the scalar field equation.
\end{abstract}

\maketitle

\section{Introduction}
In this paper, we are concerned with the following
diffusive Lotka-Volterra cooperative model
with density-dependent nonlinear diffusion terms:
\begin{equation}\label{para}
\begin{cases}
u_{t}=d_{1}\Delta u+\alpha\nabla\cdot\biggl[
v^{2}\,\nabla\biggl(\dfrac{u}{v}\biggr)\biggr]
+u(a_{1}-b_{1}u+c_{1}v),\ \ \ 
&(x,t)\in\Omega\times (0,T),\\ 
v_{t}\,=d_{2}\Delta v +\beta\nabla\cdot\biggl[
u^{2}\,\nabla\biggl(\dfrac{v}{u}\biggr)\biggr]
+v(-a_{2}+b_{2}u-c_{2}v),\ \ \ 
&(x,t)\in\Omega\times (0,T),\\
\partial_{\nu}u=
\partial_{\nu}v=0,\ \ \ 
&(x,t)\in\partial\Omega\times (0,T),\\
u(x,0)=u_{0}(x)\ge 0,\ \ 
v(x,0)=v_{0}(x)\ge 0,\ \ \ 
&x\in\Omega,
\end{cases}
\end{equation}  
where 
$\Omega\,(\subset\mathbb{R}^{N})$ is 
a bounded domain with a smooth boundary
$\partial\Omega$;
unknown functions $u(x,t)$ and $v(x,t)$ are 
the population densities at location $x\in\Omega$ and time $t>0$ 
of two species in a symbiotic relationship with each other. 
We assume the homogeneous Neumann boundary conditions that 
each species has the zero derivative in the direction of 
the outward unit normal vector $\nu$ on the boundary 
$\partial\Omega$ of the habitat $\Omega$.
The coefficients $a_i$, $b_i$ and $c_i$
$(i=1,2)$ in the reaction terms are all positive constants.
The growth rate of the species corresponding to $u$ 
is $a_{1}>0$ assuming natural increase, 
while that of the species corresponding to $v$ 
is $-a_{2}<0$ assuming natural decrease.
The positive coefficients $c_1$ and $b_2$ indicates
the magnitude of symbiotic interactions between different species.
The positive constants
$d_1$ and $d_2$ 
stand for random diffusion rates of each individual
of the species.
The strongly coupled diffusion terms 
$\nabla\cdot
[v^{2}\nabla (u/v)]$
and
$\nabla\cdot
[u^{2}\nabla (v/u)]$
describe a situation in which 
each species migrates with higher probability 
to areas with higher densities of different species.
In terms of the diffusion process in ecology,
the strongly coupled diffusion term microscopically models
a situation where the transition probability of
each individual of each species depends on the density
of the other species at the point of arrival
(Okubo and Levin
\cite[Section 5.4]{OL}).
Although the strongly coupled diffusion term is placed in \cite{OL} 
as a prototype of bio-diffusion, 
along with the cross-diffusion term $\Delta (uv)$
and the chemotaxis term $\nabla\cdot (u\nabla v)$, 
there have not been many studies from the perspective 
of reaction-diffusion systems.
Concerning a prey-predator model with the
strongly coupled diffusion term as in \eqref{para},
Oeda and the second author \cite{OK, KO} studied the set of
stationary solutions.
Furthermore, for the prey-predator model,
Heihoff and Yokota \cite{HY} established
the global existence and boundedness of time-depending
solutions and derived the convergence of solutions to the 
positive constant solution.

Throughout this paper, we assume the
weak cooperative condition
\begin{equation}\label{wc}
\dfrac{c_{1}}{c_{2}}<\dfrac{b_{1}}{b_{2}}<\dfrac{a_{1}}{a_{2}}.
\end{equation}
Under the condition \eqref{wc},
it is easy to check that 
\eqref{para} without the initial conditions
admits the positive constant solution
$$
(u^{*}, v^{*})=\biggl(
\dfrac{a_{1}c_{2}-a_{2}c_{1}}{b_{1}c_{2}-b_{2}c_{1}},
\dfrac{a_{1}b_{2}-a_{2}b_{1}}{b_{1}c_{2}-b_{2}c_{1}}
\biggr),
$$
and moreover, $(u^{*}, v^{*})$ is 
globally asymptotically stable in the kinetic system
of the corresponding ordinary differential equation
with $d_{i}=\alpha=\beta=0$
(e.g., Goh \cite{Go}).
Furthermore, also in the linear diffusion case where
$d_{i}>0$ and $\alpha =\beta =0$,
the well-known result by Weinberger \cite{We}
on the invariant region of a class of
reaction-diffusion system ensures that
$(u^{*}, v^{*})$ is globally asymptotically stable
in the sense that all positive solutions of \eqref{para}
converge to $(u^{*}, v^{*})$
uniformly in $\overline{\Omega}$ as $t\to\infty$.

Therefore, it is natural to ask whether the
strongly coupled diffusion term 
$\nabla\cdot
[v^{2}\nabla (u/v)]$
or
$\nabla\cdot
[u^{2}\nabla (v/u)]$
can produce nonconstant
steady-states and the nontrivial spatio-temporal 
dynamics instead of the global stability of 
$(u^{*}, v^{*})$.
With those objectives in mind
this paper focuses on the effect of the strongly coupled diffusion term
on the set of
stationary solutions.
Then we study
the stationary problem which consists of the nonlinear
elliptic equations
\begin{subequations}\label{st}
\begin{align}
\label{sta}
&d_{1}\Delta u+\alpha\nabla\cdot\biggl[
v^{2}\,\nabla\biggl(\dfrac{u}{v}\biggr)\biggr]
+u(a_{1}-b_{1}u+c_{1}v)=0
\ &\mbox{in}\ \Omega,\\
\label{stb}
&d_{2}\Delta v +\beta\nabla\cdot\biggl[
u^{2}\,\nabla\biggl(\dfrac{v}{u}\biggr)\biggr]
+v(-a_{2}+b_{2}u-c_{2}v)=0
\ &\mbox{in}\ \Omega,
\end{align}
subject to the homogeneous Neumann boundary conditions
\begin{equation}\label{stc}
\partial_{\nu}u=\partial_{\nu}v=0\ \ \mbox{on}\ \partial\Omega
\end{equation}
and the nonnegative conditions
\begin{equation}\label{std}
u\ge 0\ \ \mbox{and}\ \ v\ge 0\ \ \mbox{in}\ \Omega.
\end{equation}
\end{subequations}
In what follows,
we call $(u,v)$ a positive solution
if $(u,v)$ satisfies \eqref{sta}-\eqref{stc}
and
$u>0$ and $v>0$ in $\Omega$.
Hence 
a positive solution corresponds to a coexistence
steady-state of the two species.

In order to find nonconstant solutions of \eqref{st},
we resort to the bifurcation theory established by
Crandall and Rabinowitz \cite{CR} regarding $d_{2}$ as the bifurcation
parameter.
The first result (Theorem \ref{bifthm}) asserts that
if $(\alpha, \beta )\in\mathcal{R}_{j}$ as in Fig.1
for each $j\in\mathbb{N}$,
then there appears a bifurcation point $d^{(j)}_{*}$ 
in the sense that
a local branch of nonconstant solutions bifurcate from
$(u^{*}, v^{*})$ at 
$d_{2}=d^{(j)}_{*}$.
Next the asymptotic behavior of nonconstant solutions of 
\eqref{st} as $\alpha\to\infty$ and/or $\beta\to\infty$
will be discussed.
The second result (Theorem \ref{limthm}) asserts that
if $\alpha=\alpha_{n}\to\infty$ and/or
$\beta=\beta_{n}\to\infty$ with 
$\alpha_{n}/\beta_{n}\to\gamma\in [0,\infty ]$,
then
any sequence of positive nonconstant solutions
$(u_{n}, v_{n})$ converges to
$(\tau^{*}v, v)$ with some positive function $v$
and $\tau^{*}=u^{*}/v^{*}$
passing to a subsequence if necessary.
Furthermore, it will be shown that
the limit function $v$ can be classified 
depending on 
$\gamma:=\lim_{n\to\infty}\alpha_{n}/\beta_{n}$:
If $\gamma\in [0,A\tau^{*})$ with $A:=a_{1}/a_{2}$,
then the limit $v$ is equal to the constant $v^{*}$,
that is, any sequence $\{(u_{n}, v_{n})\}$
converges to $(u^{*}, v^{*})$,
whereas if $\gamma\in (A\tau^{*},\infty]$,
then $\{(u_{n}, v_{n})\}$ approaches a positive solution
of the scalar
field equation.
Consequently, it can be said from this result that
nonconstant solutions of \eqref{st} can be 
approximated by some solution of the
scalar field equation when the strongly coupled
diffusion terms are very strong.
 
In this paper, the eigenvalue problem of the Laplace equation:
$$
-\Delta u=\lambda u\quad\mbox{in}\ \Omega,\qquad
\partial_{\nu}u=0
\quad\mbox{on}\ \partial\Omega
$$
will play an important role.
We denote the all eigenvalues by
$$0=\lambda_{0}<\lambda_{1}\le\lambda_{2}\le\cdots\le\lambda_{j}\le\cdots$$
counting multiplicity and
denote by
$\varPhi_{j}$ any $L^{2}(\Omega )$ normalized eigenfunciton
corresponding to the eigenvalue $\lambda_{j}$.
Furthermore, the usual norm of $L^{p}(\Omega )$ will be denoted by
$$
\|u\|_{p}:=\biggl(
\int_{\Omega}|u(x)|^{p}\biggr)^{1/p}\quad\mbox{if}\ 1\le p<\infty$$
and
$\|u\|_{\infty}=\sup_{x\in\Omega}|u(x)|$.
As typical functional spaces for the
bifurcation analysis concerning \eqref{st}, 
we define
$$X_{p}:=W^{2,p}_{\nu }(\Omega )=\{\,u\in W^{2,p}(\Omega )\,:\,
\partial_{\nu}u=0\ \mbox{on}\ \partial\Omega\,\}$$
and
$$\boldsymbol{X}_{p}:=X_{p}\times X_{p},\quad
\boldsymbol{Y}_{p}:=L^{p}(\Omega )\times L^{p}(\Omega )$$
for any $p>N$.
\section{Main results}
The first result asserts that
bifurcation points appear on the branch 
$$\varGamma_{0}:=\{\,(d_{2},u^{*},v^{*})
\in\mathbb{R}_{+}\times\boldsymbol{X}_{p}\,\},
\quad\mbox{where}\quad \mathbb{R}_{+}:=(0,\infty),$$
of
the positive constant solution $(u^{*}, v^{*})$
by the effect of
the strongly coupled diffusion terms.
\begin{thm}\label{bifthm}
Assume \eqref{wc}.
Let $\lambda_{j}$ with $j\ge 1$ be any simple eigenvalue.
Suppose that 
\begin{equation}\label{Rj}
(\alpha, \beta )\in 
\mathcal{R}_{j}:=
\{\,(\alpha, \beta )\in\mathbb{R}_{\ge 0}^{2}\,:\,
a_{2}v^{*}\lambda_{j}\alpha-
(a_{1}+d_{1}\lambda_{j})u^{*}\lambda_{j}\beta
-(c_{2}d_{1}\lambda_{j}
+a_{1}c_{2}-a_{2}c_{1})v^{*}>0\,\},
\end{equation}
where $\mathbb{R}_{\ge 0}:=[0,\infty )$.
Then, a simple curve of nonconstant positive solutions
of \eqref{st} bifurcates from the branch $\varGamma_{0}$ of 
the positive constant solution at 
\begin{equation}\label{djdef}
d_{2}=d_{*}^{(j)}(\alpha, \beta):=
\dfrac{a_{2}v^{*}\lambda_{j}\alpha-
(a_{1}+d_{1}\lambda_{j})u^{*}\lambda_{j}\beta
-(c_{2}d_{1}\lambda_{j}
+a_{1}c_{2}-a_{2}c_{1})v^{*}}
{\{\,(d_{1}+\alpha v^{*})\lambda_{j}+b_{1}u^{*}\,\}\lambda_{j}}.
\end{equation}
To be precise,
there exist a neighborhood $\mathcal{U}_{j}
\subset\mathbb{R}\times\boldsymbol{X}_{p}$ of
$(d^{(j)}_{*}, u^{*}, v^{*})$ and
a small positive number $\delta_{j}$ such that
all solutions of \eqref{st} 
(treating $d_{2}$ as a
positive parameter) contained in $\mathcal{U}_{j}$
consist of the union of
$\varGamma_{0}\cap\mathcal{U}_{j}$ and
a simple curve
$$
\varGamma_{j}\,:\,
\left[
\begin{array}{c}
d_{2}\\
u\\
v
\end{array}
\right]=
\left[
\begin{array}{l}
d^{(j)}_{*}(\alpha, \beta)\\
u^{*}\\
v^{*}
\end{array}
\right]
+
\left[
\begin{array}{l}
q(s)\\
s(\varPhi_{j}+\widetilde{u}(s))\\
s(\kappa_{j}\varPhi_{j}+\widetilde{v}(s))\\
\end{array}
\right]
\quad\mbox{for}\ s\in (-\delta_{j}, \delta_{j}),
$$
where $\kappa_{j}=
(b_{1}u^{*}+(d_{1}+\alpha v^{*})\lambda_{j})/
(c_{1}+\lambda_{j}\alpha )u^{*}$
and
$(q(s), \widetilde{u}(s), \widetilde{v}(s))
\in\mathbb{R}\times\boldsymbol{X}_{p}$
are smooth functions satisfying
$(q(0), \widetilde{u}(0), \widetilde{v}(0))
=(0,0,0)$ and
$\int_{\Omega}\varPhi_{j}\widetilde{u}(s)=
\int_{\Omega}\varPhi_{j}\widetilde{v}(s)=0
$ for $s\in (-\delta_{j}, \delta_{j})$.
\end{thm}

\begin{figure}
\begin{center}
{\includegraphics*[scale=.5]{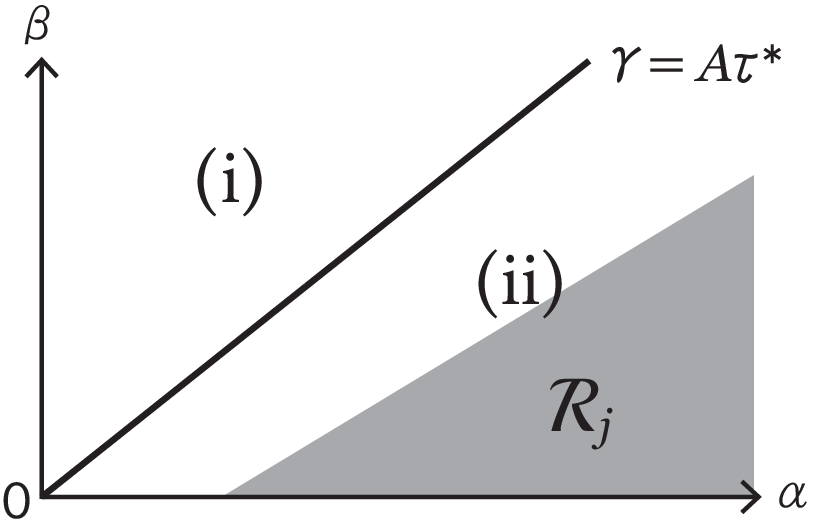}}\\
\caption{$\mathcal{R}_{j}$ and $\gamma =A\tau^{*}$ on the $\alpha\beta$ plane}
\end{center}
\end{figure}

See Figure 1 for a depiction of the region 
$\mathcal{R}_{j}$ on the $\alpha\beta$ plane.

The next result asserts that the limit
of positive solutions of \eqref{st} as
$\alpha\to\infty$ and/or $\beta\to\infty$
with $\alpha/\beta\to\gamma\in [0,\infty ]$
can be characterized by either 
$(u^{*}, v^{*})$ or
a positive solution to
the scalar field equation depending on $\gamma$.

\begin{thm}\label{limthm}
Assume \eqref{wc} and $N\le 3$.
Let $\{(\alpha_{n}, \beta_{n})\}\subset\mathbb{R}_{\ge 0}^{2}$
be any nonnegative sequence satisfying 
$\alpha_{n}/\beta_{n}\to\gamma\in [0,\infty ]$
and at least one of $\{\alpha_{n}\}$ and $\{\beta_{n}\}$
tends to $\infty$.
Let $\{(u_{n}, v_{n})\}$ be any sequence of 
positive nonconstant solutions of \eqref{st}
with $(\alpha, \beta )=(\alpha_{n}, \beta_{n})$.
Then there exists a positive function 
$v\in C^{2}(\overline{\Omega })$
such that
$$
\lim_{n\to\infty} (u_{n}, v_{n})=(\tau^{*}, 1)v
\quad\mbox{in}\ C^{1}(\overline{\Omega })\times
C^{1}(\overline{\Omega })
\quad\mbox{with}\ \ \tau^{*}=\dfrac{u^{*}}{v^{*}},$$
passing to
a subsequence if necessary.
Furthermore, the following properties hold true:
\begin{enumerate}[{\rm (i)}]
\item
If $\gamma\in [0, A\tau^{*})$ with $A=a_{1}/a_{2}$,
then $v=v^{*}$, more precisely in this case,
it follows that
$\lim_{n\to\infty}(u_{n}, v_{n})= (u^{*}, v^{*})$ in
$C^{1}(\overline{\Omega})\times C^{1}(\overline{\Omega })$
with the full sequence.
\item
If $\gamma\in (A\tau^{*}, \infty]$, then 
$v$ satisfies
\begin{equation}\label{lim}
\begin{cases}
d\Delta v+
\xi^{*}v(v-v^{*})
=0\quad&\mbox{in}\ \Omega,\\
\partial_{\nu}v=0
\quad&\mbox{on}\ \partial\Omega,
\end{cases}
\end{equation}
where
$$
(d, \xi^{*}):=
\begin{cases}
(
\tau^{*}d_{1}+\gamma d_{2},
(\gamma a_{2}-\tau^{*}a_{1})/v^{*})
\quad&\mbox{if}\ \gamma<\infty,\\
(d_{2}, a_{2}/v^{*})
\quad&\mbox{if}\ \gamma =\infty.
\end{cases}
$$
\end{enumerate}
\end{thm}
The limit equation \eqref{lim} can be classified into 
the scalar field equation, 
and the rich structure 
of nonconstant solutions with small $d_2>0$ 
has been actively studied in the field of nonlinear
elliptic equations (e.g., \cite[Chapter 3]{Ni}).
Therefore, the assertion (ii) of Theorem \ref{limthm} 
naturally enables us to expect that
\eqref{para} has also a rich structure of
nonconstant stationary patterns when 
$\alpha$ is large enough such as
$\alpha/\beta >A\tau^{*}$ and 
$d_{2}>0$ is very small.
See Figure 1 for the regions (i) and (ii) occur on the 
$\alpha\beta$ plane.

\section{A priori estimates}
In this section, some a priori estimates of
solutions of \eqref{st} will be shown.
We begin with the $L^{2}$ estimate independently
of $d_{i}$ 
$(i=1,2)$,
$\alpha$
and 
$\beta$.
\begin{lem}\label{L2lem}
Assume \eqref{wc}.
Let $(u,v)$ be any solution of \eqref{st}.
Then it holds that
$$
\|u\|_{2}\le\dfrac{a_{1}c_{2}}
{b_{1}c_{2}-b_{2}c_{1}}
|\Omega |^{1/2}\quad\mbox{and}\quad
\|v\|_{2}\le\dfrac{a_{1}b_{2}}
{b_{1}c_{2}-b_{2}c_{1}}
|\Omega |^{1/2}.
$$
\end{lem}

\begin{proof}
It is noted that \eqref{sta} is equivalent to
\begin{equation}\label{divu}
\nabla\cdot
[(d_{1}+\alpha v)\nabla u -\alpha u\nabla v]
+u(a_{1}-b_{1}u+c_{1}v)=0
\end{equation}
in which the linear and nonlinear diffusion 
terms are expressed as the divergence form.
Integrating \eqref{divu} over $\Omega$ and 
using the divergence theorem, 
one can get
$$
\int_{\Omega}u(a_{1}-b_{1}u+c_{1}v)=
\int_{\partial\Omega}
[(d_{1}+\alpha v)\nabla u -\alpha u\nabla v]
\cdot\nu\,{\rm d}\sigma
=0
$$
by the Neumann boundary conditions on $u$ and $v$.
It follows that
$b_{1}\|u\|^{2}_{2}=a_{1}\|u\|_{1}+c_{1}\int_{\Omega}uv$.
Applying the Schwarz inequality to both terms in the right-hand side, we obtain
\begin{equation}\label{Sch1}
b_{1}\|u\|_{2}\le a_{1}|\Omega |^{1/2}+c_{1}\|v\|_{2}.
\end{equation}
Hence the same procedure for \eqref{stb} gives
\begin{equation}\label{Sch2}
c_{2}\|v\|_{2}\le b_{2}\|u\|_{2}.
\end{equation}
By \eqref{Sch1} and \eqref{Sch2},
one can get the desired estimates on $\|u\|_{2}$ and $\|v\|_{2}$.
The proof of Lemma \ref{L2lem} is complete.
\end{proof}
Furthermore, in the case where the spatial dimension
$N$ satisfies $N\le 3$, we find the uniform $W^{2,p}$ 
bound 
which is independent of $\alpha$ and $\beta$ as follows:
\begin{thm}\label{aprthm}
Assume \eqref{wc} and $N\le 3$.
Then for any $p\in (1,\infty )$,
there exists a positive number $C_{p}=C_{p}(a_{i}, b_{i}, c_{i}, d_{i})$,
which is independent of $\alpha $ and $\beta$, 
such that
any solution $(u,v)$ of \eqref{st} satisfies
$$\|u\|_{W^{2,p}(\Omega )}\le C_{p}
\quad\mbox{and}\quad
\|v\|_{W^{2,p}(\Omega )}\le C_{p}.
$$
\end{thm}

In the proof of Theorem \ref{aprthm},
the following Harnack inequality
will play an important role.

\begin{lem}[\cite{LN2}]\label{Hlem}
Let $w$ be a nonnegative solution of
$$\Delta w+V(x)w=0\quad\mbox{in}\ \Omega,\qquad
\partial_{\nu}w=0\quad\mbox{on}\ \partial\Omega.$$
Then for any $p>\max\{N/2, 1\}$,
there exists $C_{\sharp}=C_{\sharp}(\|V\|_{p})>0$
such that
$$
\max_{x\in\overline{\Omega}}w(x)\le C_{\sharp}\min_{x\in\overline{\Omega}}
w(x).
$$
\end{lem}

\begin{proof}[Proof of Theorem \ref{aprthm}]
Since $\nabla\cdot (v\nabla u-u\nabla v)=v\Delta u-u\Delta v$,
then \eqref{divu}
(or \eqref{sta})
is reduced to
\begin{equation}\label{st1}
(d_{1}+\alpha v)\Delta u-\alpha u\Delta v+u(a_{1}-b_{1}u+c_{1}v)=0
\quad\mbox{in}\ \Omega.
\end{equation}
Similarly one can form \eqref{stb} derive
\begin{equation}\label{st2}
(d_{2}+\beta u)\Delta v-\beta v\Delta u+v(-a_{2}+b_{2}u-c_{2}v)=0
\quad\mbox{in}\ \Omega.
\end{equation}
Substituting $\Delta v$ in \eqref{st2} for $\Delta v$  in \eqref{st1}, 
and substituting $\Delta u$ in \eqref{st1} for $\Delta u$ in \eqref{st2},
we obtain the following semilinear system which is equivalent to
\eqref{st}:
\begin{equation}\label{semi}
\begin{cases}
\Delta u + V_{1}(u,v; \alpha, \beta )u=0,\quad u\ge 0
\quad&\mbox{in}\ \Omega,\\
\Delta v+ V_{2}(u,v; \alpha, \beta )v=0,\quad v\ge 0
\quad&\mbox{in}\ \Omega,\\
\partial_{\nu}u=\partial_{\nu }v=0\quad&\mbox{on}\ \partial\Omega,
\end{cases}
\end{equation}
where
\begin{equation}\label{V1}
V_{1}(u,v; \alpha, \beta ):=
\dfrac{(d_{2}+\beta u)(a_{1}-b_{1}u+c_{1}v)
+\alpha v (-a_{2}+b_{2}u-c_{2}v)}
{d_{1}d_{2}+d_{1}\beta u+d_{2}\alpha v}
\end{equation}
and
\begin{equation}\label{V2}
V_{2}(u,v; \alpha, \beta ):=
\dfrac{(d_{1}+\alpha v)(-a_{2}+b_{2}u-c_{2}v)
+\beta u (a_{1}-b_{1}u+c_{1}v)}
{d_{1}d_{2}+d_{1}\beta u+d_{2}\alpha v}.
\end{equation}
Our strategy of the proof is first 
to find the uniform bound of 
$\|V_{i}(u,v;\alpha, \beta)\|_{2}$
independently of $(\alpha, \beta )$
with the aid of Lemma \ref{L2lem} 
and next to apply Lemma \ref{Hlem} to \eqref{semi}.
It follows from \eqref{V1} that
\begin{equation}\label{V11}
\begin{split}
|V_{1}(u,v;\alpha, \beta )|&\le
\left|
\dfrac{(d_{2}+\beta u)(a_{1}-b_{1}u+c_{1}v)}
{d_{1}d_{2}+d_{1}\beta u+d_{2}\alpha v}\right|+
\left|
\dfrac{\alpha v (-a_{2}+b_{2}u-c_{2}v)}
{d_{1}d_{2}+d_{1}\beta u+d_{2}\alpha v}\right|\\
&\le
\dfrac{|a_{1}-b_{1}u+c_{1}v|}
{d_{1}}+
\dfrac{|-a_{2}+b_{2}u-c_{2}v|}{d_{2}}.
\end{split}
\end{equation}
Similarly, we know from \eqref{V2} that
\begin{equation}\label{V22}
|V_{2}(u,v;\alpha, \beta )|\le
\dfrac{|-a_{2}+b_{2}u-c_{2}v|}{d_{2}}+
\dfrac{|a_{1}-b_{1}u+c_{1}v|}{d_{1}}.
\end{equation}
Hence \eqref{V11} and \eqref{V22} with Lemma \ref{L2lem}
ensure a positive constant
$\overline{C}=\overline{C}(a_{i}, b_{i}, c_{i}, d_{i})$
independent of $(\alpha, \beta )$ such that
$\|V_{i}(u,v; \alpha, \beta )\|_{2}\le\overline{C}$
for $i=1$,
$2$.
Therefore, if $N\le 3$, then the application of Lemma \ref{Hlem} 
for \eqref{semi} with $p=2$ ensures some $C_{\sharp}=
C_{\sharp}(a_{i}, b_{i}, c_{i}, d_{i})$ such that
$$
\max_{x\in\overline{\Omega}}u(x)\le
C_{\sharp}\min_{x\in\overline{\Omega}}u(x),
\qquad
\max_{x\in\overline{\Omega}}v(x)\le
C_{\sharp}\min_{x\in\overline{\Omega}}v(x).
$$
Since
$\min_{x\in\overline{\Omega}}u(x)\le a_{1}c_{2}/
(b_{1}c_{2}-b_{2}c_{1})$ 
and
$\min_{x\in\overline{\Omega}}v(x)\le a_{1}b_{2}/
(b_{1}c_{2}-b_{2}c_{1})$ by Lemma \ref{L2lem},
then we obtain the uniform $L^{\infty}$ bound,
which is independent of $\alpha$ and $\beta$,
as follows:
\begin{equation}\label{Linf}
\max_{x\in\overline{\Omega }}u(x)\le
M\quad\mbox{and}\quad
\max_{x\in\overline{\Omega }}v(x)\le M
\quad\mbox{where}\quad
M:=C_{\sharp}\dfrac{a_{1}\max\{ b_{2}, c_{2}\}}
{b_{1}c_{2}-b_{2}c_{1}}.
\end{equation}
By virtue of \eqref{semi},
the standard application of the elliptic regularity theory
using
\eqref{V11}-\eqref{Linf} ensures
the desired bound $C_{p}$ in the assertion.
The proof of Theorem \ref{aprthm} is complete.
\end{proof}

\begin{thm}\label{nonexthm}
Let $M$ be the positive constant obtained in \eqref{Linf}.
If \eqref{st} admits at least one positive nonconstant solution,
then
$$
d_{1}<\dfrac{(\alpha +\beta )M}{2}+
\dfrac{1}{\lambda_{1}}
\biggl(
a_{1}+
\dfrac{(3c_{1}+b_{2})M}{2}
\biggr)
$$
or
$$
d_{2}<\dfrac{M}{2}\biggl(
\alpha +\beta +
\dfrac{c_{1}+3b_{2}}{\lambda_{1}}\biggr),
$$
\end{thm}

\begin{proof}
Suppose that $(u,v)$ is a positive nonconstant solution of \eqref{st}.
We set $\overline{u}:=|\Omega|^{-1}\int_{\Omega}u$ and 
$\overline{v}:=|\Omega|^{-1}\int_{\Omega }v$.
Taking the inner product of \eqref{divu} with $u-\overline{u}$,
we obtain
\begin{equation}\label{inner}
\int_{\Omega}
(d_{1}+\alpha v)|\nabla u|^{2}
-\alpha\int_{\Omega}u\nabla u\cdot\nabla v
=\int_{\Omega }(u-\overline{u})
u(a_{1}-b_{1}u-c_{1}v).
\end{equation}
Here we observe that
\begin{equation}\label{r1}
a_{1}\int_{\Omega }(u-\overline{u})u=
a_{1}\int_{\Omega }(u-\overline{u})(u-\overline{u}+\overline{u})=
a_{1}\|u-\overline{u}\|^{2}_{2}
\end{equation}
and
\begin{equation}\label{r2}
-b_{1}\int_{\Omega }(u-\overline{u})u^{2}=
-b_{1}\int_{\Omega }(u-\overline{u})
(u^{2}-\overline{u}^{2}+\overline{u}^{2})=
-b_{1}\int_{\Omega}(u-\overline{u})^{2}(u+\overline{u})
\end{equation}
and
\begin{equation}\label{r3}
\begin{split}
c_{1}\int_{\Omega}
(u-\overline{u})uv&=
c_{1}\int_{\Omega}(u-\overline{u})(u-\overline{u}+\overline{u})
v\\
&=
c_{1}\int_{\Omega}
(u-\overline{u})^{2}v+
c_{1}\overline{u}\int_{\Omega}
(u-\overline{u})v\\
&=
c_{1}\int_{\Omega}
(u-\overline{u})^{2}v+
c_{1}\overline{u}\int_{\Omega}
(u-\overline{u})(v-\overline{v}).
\end{split}
\end{equation}
Substituting \eqref{r1}-\eqref{r3} into \eqref{inner}, 
one can see that
$$
d_{1}\|\nabla u\|^{2}
-\alpha\int_{\Omega}u\nabla u\cdot\nabla v\\
<
a_{1}\|u-\overline{u}\|^{2}_{2}
+
c_{1}\int_{\Omega}
(u-\overline{u})^{2}v+
c_{1}\overline{u}\int_{\Omega}
(u-\overline{u})(v-\overline{v}).
$$
Therefore,
together with \eqref{Linf},
the Schwarz and the Young inequalities yield
\begin{equation}\label{add1}
\biggl(d_{1}-\dfrac{\alpha M}{2}\biggr)\|\nabla u\|^{2}_{2}-
\dfrac{\alpha M}{2}\|\nabla v\|^{2}_{2}<
\biggl(a_{1}+\dfrac{3c_{1}M}{2}\biggr)\|u-\overline{u}\|^{2}_{2}
+\dfrac{c_{1}M}{2}\|v-\overline{v}\|^{2}_{2}.
\end{equation}
Obviously, the same procedure for \eqref{stb} yields
\begin{equation}\label{add2}
\biggl(d_{2}-\dfrac{\beta M}{2}\biggr)\|\nabla v\|^{2}_{2}-
\dfrac{\beta M}{2}\|\nabla u\|^{2}_{2}<
\dfrac{b_{2}M}{2}\|u-\overline{u}\|^{2}_{2}+
\dfrac{3b_{2}M}{2}\|v-\overline{v}\|^{2}_{2}.
\end{equation}
Adding \eqref{add1} and \eqref{add2}, one can see 
\begin{equation}
\begin{split}
&\biggl(
d_{1}-\dfrac{(\alpha +\beta )M}{2}\biggr)\|\nabla u\|^{2}_{2}+
\biggl(
d_{2}-\dfrac{(\alpha +\beta )M}{2}\biggr)\|\nabla v\|^{2}_{2}\\
<&
\biggl(a_{1}+\dfrac{(3c_{1}+b_{2})M}{2}\biggr)\|u-\overline{u}\|^{2}_{2}+
+\dfrac{(c_{1}+3b_{2})M}{2}\|v-\overline{v}\|^{2}_{2}.
\end{split}
\nonumber
\end{equation}
Together with the Poincar\'e-Wirtinger equality, we obtain
\begin{equation}
\begin{split}
&\biggl\{\,
d_{1}-\dfrac{(\alpha+\beta )M}{2}-\dfrac{1}{\lambda_{1}}
\biggl( a_{1}+\dfrac{(3c_{1}+b_{2})M}{2}\biggr)\,\biggr\}
\|\nabla u\|^{2}_{2}\\
&+
\biggl\{\,
d_{2}-\dfrac{(\alpha +\beta )M}{2}-\dfrac{(c_{1}+3b_{2})M}{2\lambda_{1}}
\,\biggr\}\|\nabla v\|^{2}_{2}< 0.
\end{split}
\nonumber
\end{equation}
Consequently, we obtain the assertion of Theorem \ref{nonexthm}.
\end{proof}

\section{Bifurcation from the constant state}
In this section, we give the proof of Theorem \ref{bifthm}.
The proof relies on the bifurcation theorem from simple
eigenvalues established by Crandall and Rabinowitz
\cite[Theorem 1.7]{CR} and its improvement by
Shi and Wang \cite[Theorem 4.3]{SW} for the
application to a class of quasilinear systems.
We begin with the linearized stability of $(u^{*},v^{*})$
when $d_{2}>0$ is sufficiently large.
\begin{lem}\label{linlem}
Assume \eqref{wc}.
Then for any $(\alpha,\beta)\in\mathbb{R}_{+}^{2}$,
there exists $\widehat{d}_{*}(\alpha, \beta )\ge 0$ such that
if $d_{2}>\widehat{d}_{*}(\alpha, \beta )$, then
$(u^{*}, v^{*})$ is linearly stable.
\end{lem}

\begin{proof}
In view of the left-hand sides of \eqref{st1} and \eqref{st2},
for any $(d_{2}, u,v)\in \mathbb{R}_{+}\times\boldsymbol{X}_{p}$ 
with $p>N$,
we define $F(d_{2}, u,v)\in\boldsymbol{Y}_{p}$ by
\begin{equation}\label{Fdef}
F(d_{2}, u, v):=
\left[
\begin{array}{ll}
d_{1}+\alpha v & -\alpha u\\
-\beta v & d_{2}+\beta u
\end{array}
\right]
\left[
\begin{array}{c}
\Delta u\\
\Delta v
\end{array}
\right]
+
\left[
\begin{array}{l}
u(a_{1}-b_{1}u+c_{1}v)\\
v(-a_{2}+b_{2}u-c_{2}v)
\end{array}
\right].
\end{equation}
In order to find bifurcation points on $\varGamma_{0}$,
we consider the linearized eigenvalue problem around
the positive constant solution $(u^{*}, v^{*})$ to
\eqref{st} as follows:
\begin{equation}\label{ei}
L(d_{2})\left[
\begin{array}{c}
\phi\\
\psi
\end{array}
\right]
+\mu
\left[
\begin{array}{c}
\phi\\
\psi
\end{array}
\right]=
\left[
\begin{array}{c}
0\\
0
\end{array}
\right],
\end{equation}
where $L(d_{2})=D_{(u,v)}F(d_{2},u^{*},v^{*})$.
By straightforward calculation, one can see that
the eigenvalue problem \eqref{ei} is reduced to
the following system of linear elliptic equations:
\begin{equation}\label{ei2}
\left[
\begin{array}{ll}
d_{1}+\alpha v^{*} & -\alpha u^{*}\\
-\beta v^{*} & d_{2}+\beta u^{*}
\end{array}
\right]
\left[
\begin{array}{c}
\Delta\phi\\
\Delta \psi
\end{array}
\right]
+
\left[
\begin{array}{ll}
\mu-b_{1}u^{*} & c_{1}u^{*}\\
b_{2}v^{*} & \mu -c_{2}v^{*}
\end{array}
\right]
\left[
\begin{array}{c}
\phi\\
\psi
\end{array}
\right]
=
\left[
\begin{array}{c}
0\\
0
\end{array}
\right].
\end{equation}
We seek for solution $(\phi, \psi)$ of \eqref{ei2}
in the following form of the Fourier expansion:
$$
(\phi, \psi )=\sum^{\infty}_{j=0}
(h_{j}, k_{j})\varPhi_{j}.
$$
By the fact that
$\{\varPhi_{j}\}$ is a complete orthogonal system in $L^{2}(\Omega )$,
we substitute the above expression into \eqref{ei2} to get
\begin{equation}\label{alg}
\left(\mu I -A_{j}(d_{2}; \alpha, \beta)\right)
\left[
\begin{array}{c}
h_{j}\\
k_{j}
\end{array}
\right]
=
\left[
\begin{array}{c}
0\\
0
\end{array}
\right],
\end{equation}
where
$$
A_{j}(d_{2}; \alpha, \beta ):=
\left[
\begin{array}{ll}
b_{1}u^{*}+(d_{1}+\alpha v^{*})\lambda_{j} &
-(c_{1}+\lambda_{j}\alpha ) u^{*}\\
-(b_{2}+\lambda_{j}\beta ) v^{*} &
c_{2}v^{*}+(d_{2}+\beta u^{*})\lambda_{j}
\end{array}
\right].
$$
It is noted that
$\mu$ is an eigenvalue of \eqref{ei} if and only if
\eqref{alg}
admits a nontrivial solution $(h_{j}, k_{j})$ with 
some $j\in\mathbb{N}\cup\{0\}$, that is,
\begin{equation}\label{charaeq}
\det\left(\mu I-A_{j}(d_{2}; \alpha, \beta)\right)\\
=\mu^{2}-\mbox{tr}\,A_{j}(d_{2}; \alpha, \beta)\mu
+\det A_{j}(d_{2}; \alpha, \beta)=0.
\end{equation}
We denote  roots of
the algebraic equation $\det\left(\mu I-A_{j}(d_{2}; \alpha, \beta)\right)=0$
by
$\mu^{\pm}_{j}(d_{2}; \alpha, \beta )$ with
$$
\mbox{Re}\,\mu^{-}_{j}(d_{2}; \alpha, \beta )\le
\mbox{Re}\,\mu^{+}_{j}(d_{2}; \alpha, \beta )
\quad\mbox{and}\quad
\mbox{Im}\,\mu^{-}_{j}(d_{2}; \alpha, \beta )\le
\mbox{Im}\,\mu^{+}_{j}(d_{2}; \alpha, \beta ).
$$
Since
$$
\mbox{tr}\,A_{j}(d_{2}; \alpha, \beta )=
b_{1}u^{*}+c_{2}v^{*}+\lambda_{j}
(d_{1}+d_{2}+\alpha v^{*}+\beta u^{*})>0,
$$
then
\begin{equation}\label{mujsign}
\begin{cases}
\mu^{-}_{j}(d_{2}; \alpha, \beta )<0<\mu^{+}_{j}(d_{2}; \alpha, \beta )
\quad&\mbox{if}\ \det A_{j}(d_{2}; \alpha, \beta )<0,\\
0=\mu^{-}_{j}(d_{2}; \alpha, \beta )<\mu^{+}_{j}(d_{2}; \alpha, \beta )
\quad&\mbox{if}\ \det A_{j}(d_{2}; \alpha, \beta )=0,\\
0<\mbox{Re}\,\mu^{-}_{j}(d_{2}; \alpha, \beta )\le
\mbox{Re}\,\mu^{+}_{j}(d_{2}; \alpha, \beta )
\quad&\mbox{if}\ \det A_{j}(d_{2}; \alpha, \beta )>0.
\end{cases}
\end{equation}
By straightforward calculation, one can verify that
\begin{equation}
\begin{split}
&\det A_{j}(d_{2}; \alpha, \beta )\\
=&
\{(d_{1}+\alpha v^{*})\lambda_{j}+b_{1}u^{*}\}\lambda_{j}
d_{2}
-a_{2}v^{*}\lambda_{j}\alpha
+(a_{1}+d_{1}\lambda_{j})u^{*}\lambda_{j}\beta
+(c_{2}d_{1}\lambda_{j}+a_{1}c_{2}-a_{2}c_{1})v^{*},
\end{split}
\nonumber
\end{equation}
and therefore,
\begin{equation}\label{detsign}
\det A_{j}(d_{2}; \alpha, \beta)
\begin{cases}
<0\quad&\mbox{if}\ 
(\alpha, \beta )\in\mathcal{R}_{j}\ \ \mbox{and}\ \ 0<d_{2}<d_{*}^{(j)}(\alpha, \beta ),\\
=0\quad&\mbox{if}\ 
(\alpha, \beta )\in\mathcal{R}_{j}\ \ \mbox{and}\ \ d_{2}=d_{*}^{(j)}(\alpha, \beta ),\\
>0\quad&\mbox{if}\ 
(\alpha, \beta )\in\mathcal{R}_{j}\ \ \mbox{and}\ \ d_{2}>d_{*}^{(j)}(\alpha, \beta ),\\
\end{cases}
\end{equation}
and
\begin{equation}\label{detsei}
\det A_{j}(d_{2}; \alpha, \beta)>0\quad
\mbox{for any}\ d_{2}>0\ 
\mbox{if}\ (\alpha, \beta )\not\in\mathcal{R}_{j},
\end{equation}
where $\mathcal{R}_{j}$ and $d_{*}^{(j)}(\alpha, \beta )$ are 
defined by \eqref{Rj} and \eqref{djdef}, respectively.
In addition, it is noted that
$\det A_{0}(d_{2}; \alpha, \beta )>0$ 
for any $d_{2}>0$ and $(\alpha, \beta)\in\mathbb{R}_{\ge 0}$ 
under \eqref{wc}.

For any $(\alpha, \beta )\in\mathbb{R}_{\ge 0}^{2}$,
we define $J (\alpha, \beta )$ by the set of
$j\in\mathbb{N}$ such that
$(\alpha, \beta )\in\mathcal{R}_{j}$,
that is, 
$$J (\alpha, \beta ):=\{\,
j\in\mathbb{N}\,:\,(\alpha, \beta)\in\mathcal{R}_{j}\,\}.$$
In view of $\mathcal{R}_{j}$, one can verify that
$\sharp J(\alpha, \beta )<\infty$
for any $(\alpha, \beta)\in\mathbb{R}_{+}^{2}$.
We set
\begin{equation}
\widehat{d}_{*}(\alpha, \beta ):=
\begin{cases}
\max\{\,d^{(j)}_{*}(\alpha, \beta )\,:\,j\in J(\alpha, \beta )\,\}
\quad&\mbox{if}\ J(\alpha, \beta )\neq\emptyset,\\
0\quad&\mbox{if}\ J(\alpha, \beta )=\emptyset.
\end{cases} 
\end{equation}
We obtain from \eqref{mujsign}-\eqref{detsei} that
if $d_{2}>\widehat{d}_{*}(\alpha, \beta )$, then
$\mbox{Re}\,\mu_{j}(\alpha, \beta )>0$ for any 
$j\in\mathbb{N}\cup\{0\}$.
Consequently, it follows that 
$(u^{*}, v^{*})$ is linearly stable if $d>\widehat{d}_{*}(\alpha, \beta )$.
The proof of Lemma \ref{linlem} is complete.
\end{proof}

\begin{proof}[Proof of Theorem \ref{bifthm}]
In view of \eqref{ei} and \eqref{charaeq},
one can see that
$\mbox{Ker}\,L(d_{2})$ is nontrivial if and only if
$\det A_{j}(d_{2}; \alpha, \beta )=0$.
Therefore, it follows from \eqref{detsign} that
$\mbox{Ker}\,L(d_{2})$ is nontrivial if and only if
$d_{2}=d^{(j)}_{*}(\alpha, \beta )$,
and moreover,
$$
\mbox{Ker}\,L(d^{(j)}_{*}(\alpha, \beta ))=
\mbox{Span}\,\{\,
(1, \kappa_{j})\varPhi_{j}\,\},
$$
where $\kappa_{j}=
(b_{1}u^{*}+(d_{1}+\alpha v^{*})\lambda_{j})/
(c_{1}+\lambda_{j}\alpha )u^{*}$.
Here it is easily verified that
$$
\mbox{Ker}\,L^{*}(d^{(j)}_{*}(\alpha, \beta ))=
\mbox{Span}\,\{\,
(1, \kappa^{*}_{j})\varPhi_{j}\,\},
$$
where
$L^{*}(d^{(j)}_{*}(\alpha, \beta ))$ is the adjoint operator of
$L(d^{(j)}_{*}(\alpha, \beta ))$ and
$\kappa_{j}^{*}=
(b_{1}u^{*}+(d_{1}+\alpha v^{*})\lambda_{j})/
(b_{2}+\lambda_{j}\beta )v^{*}$.
By the Fredholm alternative theorem, one can see that
$$\mbox{Ran}\,L(d^{(j)}_{*}(\alpha, \beta ))=
\mbox{Span}\,\{\,
(1, \kappa^{*}_{j})\varPhi_{j}\,\}^{\bot},
$$
that is,
$L(d^{(j)}_{*}(\alpha, \beta ))$ is a Fredholm operator
with index zero.
For the application of
\cite[Theorem 4.3]{SW},
we have to check the transversality condition that
\begin{equation}\label{range}
D_{(u,v),d_{2}}F(d^{(j)}_{*}(\alpha, \beta ), u^{*}, v^{*})
\biggl[
\begin{array}{l}
1\\
\kappa_{j}
\end{array}
\biggr]\varPhi_{j}\not\in
\mbox{Ran}\,L(d^{(j)}_{*}(\alpha, \beta ))\,
(\,=
\mbox{Span}\,\{\,
(1, \kappa^{*}_{j})\varPhi_{j}\,\}^{\bot}\,).
\end{equation}
Here it is easily verified that
$$
D_{(u,v),d_{2}}F(d^{(j)}_{*}(\alpha, \beta ), u^{*}, v^{*})
\biggl[
\begin{array}{l}
1\\
\kappa_{j}
\end{array}
\biggr]\varPhi_{j}=
\left[
\begin{array}{l}
0\\
-\kappa_{j}\lambda_{j}
\end{array}
\right]\varPhi_{j}.
$$
Obviously,
the right-hand side is not orthogonal to 
$\mbox{Span}\,\{\,
(1, \kappa^{*}_{j})\varPhi_{j}\,\}$,
and thereby,
\eqref{range} follows.
Therefore, we can use the bifurcation theorem
\cite[Theorem 4.3]{SW} to
obtain all assertions in Theorem \ref{bifthm}.
\end{proof}

\section{Asymptotic behavior of solutions as $\alpha$, $\beta\to\infty$}
This section is devoted to the proof of Theorem \ref{limthm}.
\begin{proof}[Proof of Theorem \ref{limthm}]
Assume \eqref{wc} and $N\le 3$.
Let $\{(u_{n}, v_{n})\}$ be any sequence of 
positive nonconstant solutions of \eqref{st}
with $(\alpha, \beta )=(\alpha_{n}, \beta_{n})$
satisfying
$\gamma_{n}:=\alpha_{n}/\beta_{n}\to\gamma
\in [0,\infty]$
and at least one of $\{\alpha_{n}\}$ and $\{\beta_{n}\}$
tends to $\infty$.
By Theorem \ref{aprthm} and the Sobolev embedding theorem,
we find nonnegative
functions
$(u_{\infty}, v_{\infty})\in\boldsymbol{X}_{p}$
with $p>N$ such that
\begin{equation}\label{limc1}
\lim_{n\to\infty} (u_{n}, v_{n})=
(u_{\infty}, v_{\infty})
\quad\mbox{weakly in}\ \boldsymbol{X}_{p}
\quad\mbox{and}\quad
\mbox{strongly in}\ C^{1}(\overline{\Omega})\times
C^{1}(\overline{\Omega }),
\end{equation}
passing to a subsequence if necessary.

As the first step of the proof,
we show that $u_{\infty}>0$ and 
$v_{\infty}>0$ in $\overline{\Omega}$.
We set 
$(\widetilde{u}_{n}(x),
\widetilde{v}_{n}(x)):=(u_{n}(x)/\|u_{n}\|_{\infty},
v_{n}(x)/\|v_{n}\|_{\infty})$.
Multiplying the first equation of \eqref{semi} by
$1/\|u_{n}\|_{\infty}$, we see that
$$
\Delta\widetilde{u}_{n}+V_{1}(u_{n}, v_{n}; \alpha_{n}, \beta_{n})
\widetilde{u}_{n}=0
\quad\mbox{in}\ \Omega,\quad
\partial_{\nu}\widetilde{u}_{n}=0\quad\mbox{on}\ \partial\Omega
$$
for any $n\in\mathbb{N}$.
By virtue of \eqref{V11} and $\|\widetilde{u}_{n}\|_{\infty}=1$,
the elliptic regularity theory and the strong maximum
principle (e.g., \cite{GT}) 
give a positive function $\widetilde{u}_{\infty}\in X_{p}$
with $p>N$ and $\|\widetilde{u}_{\infty}\|_{\infty}=1$
such that
$\widetilde{u}_{n}\to\widetilde{u}_{\infty}$
weakly in $X_{p}$ and strongly in $C^{1}(\overline{\Omega})$.
Obviously, the same procedure for the second equation of \eqref{semi}
implies that $\widetilde{v}_{n}\to\widetilde{v}_{\infty}$ 
weakly in $X_{p}$ and strongly in $C^{1}(\overline{\Omega})$
with some positive function 
$\widetilde{v}_{\infty}\in X_{p}$
with $p>N$ and $\|\widetilde{v}_{\infty}\|_{\infty}=1$.
Next, multiplying 
\eqref{stb}
by 
$1/\|v_{n}\|_{\infty}$
and
integrating the resulting expression over $\Omega$,
one can see
\begin{equation}\label{intn}
\int_{\Omega}
\widetilde{v}_{n}(-a_{2}+b_{2}u_{n}-c_{2}v_{n})=0
\quad\mbox{for any}\ n\in\mathbb{N}
\end{equation}
by \eqref{stc}.
Suppose for contradiction that
$\min_{x\in\overline{\Omega}}u_{n}(x)\to 0$ as
$n\to\infty$, 
subject to a subsequence.
Thanks to \eqref{V11},
we can set $n\to\infty$ in the weak form of 
\begin{equation}\label{un}
\Delta u_{n}+V_{1}(u_{n}, v_{n}; \alpha_{n}, \beta_{n})u_{n}=0
\quad
\mbox{in}\ \Omega,\qquad
\partial_{\nu}u_{n}=0
\quad\mbox{on}\ \partial\Omega
\end{equation} 
and apply the combination of the elliptic regularity with
the strong maximum principle to reach
$u_{\infty}=0$ in $\overline{\Omega}$.
Setting $n\to\infty$ in \eqref{intn}, we know that
$$
\int_{\Omega}
\widetilde{v}_{\infty}(-a_{2}-c_{2}v_{\infty})=0.
$$
Hence this contradicts the fact that $\widetilde{v}_{\infty }>0$ in $\overline{\Omega}$.
Consequently,
we obtain $u_{\infty}>0$ in $\overline{\Omega}$ by contradiction.
We next suppose for contradiction that
$\min_{x\in\overline{\Omega}}v_{n}(x)\to 0$ as
$n\to\infty$,
passing to a subsequence.
Owing to the uniform estimate \eqref{V22},
the same limiting procedure 
in the weak form of 
\begin{equation}\label{vn}
\Delta v_{n}+V_{2}(u_{n}, v_{n}; \alpha_{n}, \beta_{n})v_{n}=0
\quad
\mbox{in}\ \Omega,\qquad
\partial_{\nu}v_{n}=0
\quad\mbox{on}\ \partial\Omega
\end{equation} 
leads to
$v_{\infty}=0$ in $\overline{\Omega}$.
In view of \eqref{V1}, setting $n\to\infty$ in \eqref{un} again, we know that
$u_{\infty}>0$ satisfies
$$
-d_{1}\Delta u_{\infty}= u_{\infty}(a_{1}-b_{1}u_{\infty})
\quad\mbox{in}\ \Omega,
\qquad\partial_{\nu}u_{\infty}=0\quad\mbox{on}\ \partial\Omega
$$
in the weak sense.
By the uniqueness of positive solutions to the above Neumann problem,
one can deduce $u_{\infty}=a_{1}/b_{1}$.
Then, setting $n\to\infty$ in \eqref{intn},
one can see
$$
\biggl(
-a_{2}+\dfrac{a_{1}b_{2}}{b_{1}}\biggr)\int_{\Omega}\widetilde{v}_{\infty}=0.
$$
From $\widetilde{v}_{\infty}>0$ in $\overline{\Omega}$, it follows that
$a_{1}/a_{2}=b_{1}/b_{2}$.
However this contradicts \eqref{wc}.
Therefore, we obtain $v_{\infty}>0$ in $\overline{\Omega}$
by contradiction.

As the next step of the proof, we show that
$u_{\infty}(x)=\tau^{*}v_{\infty}(x)$
for all $x\in\overline{\Omega}$ with
$\tau^{*}=u^{*}/v^{*}$.
Taking the inner product of \eqref{sta} 
with $u_{n}/v_{n}$, we get
\begin{equation}\label{an}
d_{1}\biggl(
\int_{\Omega }\dfrac{|\nabla u_{n}|^{2}}{v_{n}}
-
\int_{\Omega }\bigg(\dfrac{u_{n}}{v_{n}}\biggr)^{2}
\nabla u_{n}\cdot\nabla v_{n}\biggr)
+\alpha_{n}\int_{\Omega }
v_{n}^{2}\biggl|\nabla\biggl(
\dfrac{u_{n}}{v_{n}}\biggr)\biggr|^{2}=
\int_{\Omega}\dfrac{u_{n}^{2}}{v_{n}}
(a_{1}-b_{1}u_{n}+c_{1}v_{n})
\end{equation}
for all $n\in\mathbb{N}$.
By virtue of \eqref{limc1} with
$u_{\infty}>0$ and $v_{\infty}>0$ in
$\Omega$,
we multiply \eqref{an} by $1/\alpha_{n}$ and
set $n\to\infty$ to get
$$
\int_{\Omega}v_{\infty}^{2}
\biggl|
\nabla\biggl(\dfrac{u_{\infty}}{v_{\infty}}\biggr)\biggr|^{2}=0.
$$
Since $u_{\infty}>0$ and $v_{\infty}>0$ in $\Omega$,
then $u_{\infty}/v_{\infty}=\tau^{*}$ in $\Omega$
with some positive constant $\tau^{*}$.
In order to get $\tau^{*}=u^{*}/v^{*}$, 
we set $n\to\infty$ in the integration of \eqref{sta} and \eqref{stb} 
to see
$$
\begin{cases}
a_{1}\|v_{\infty}\|_{1}+(-b_{1}\tau^{*}+c_{1})\|v_{\infty}\|^{2}_{2}=0,\\
-a_{2}\|v_{\infty}\|_{1}+(b_{2}\tau^{*}-c_{2})\|v_{\infty}\|^{2}_{2}=0.
\end{cases}
$$
It follows from the algebraic equations that $\tau^{*}=u^{*}/v^{*}$.

As the final step of the proof, 
we show the desired assertions (i) and (ii) of Theorem \ref{limthm}.
By virtue of $u_{\infty}=\tau^{*}v_{\infty}$ in 
\eqref{limc1}, we set $n\to\infty$ in \eqref{un} and \eqref{vn}
using \eqref{V1} and \eqref{V2}, respectively,
to reach the same equation of $v_{\infty}$ as follows:
\begin{equation}\label{lim2}
\begin{cases}
d\Delta v_{\infty }+
\xi^{*}v_{\infty }(v_{\infty }-v^{*})
=0\quad&\mbox{in}\ \Omega,\\
\partial_{\nu}v_{\infty }=0
\quad&\mbox{on}\ \partial\Omega,
\end{cases}
\end{equation}
where
$$
(d, \xi^{*}):=
\begin{cases}
(
\tau^{*}d_{1}+\gamma d_{2},
(\gamma a_{2}-\tau^{*}a_{1})/v^{*})
\quad&\mbox{if}\ \gamma<\infty,\\
(d_{2}, a_{2}/v^{*})
\quad&\mbox{if}\ \gamma =\infty.
\end{cases}
$$
It is well-known that 
if $\xi^{*}<0$, then \eqref{lim2} is 
corresponding to the diffusive logistic equation
and $v_{\infty}=v^{*}$,
whereas if $\xi^{*}>0$, then \eqref{lim2} is
corresponding to the scalar field equation
and admits positive nonconstant 
solutions when $d>0$ is small.
The proof of Theorem \ref{limthm} is complete.
\end{proof}


\begin{figure}
\centering
\subfigure[$(\alpha, \beta )= (2,1)$]{
\includegraphics*[scale=.3]{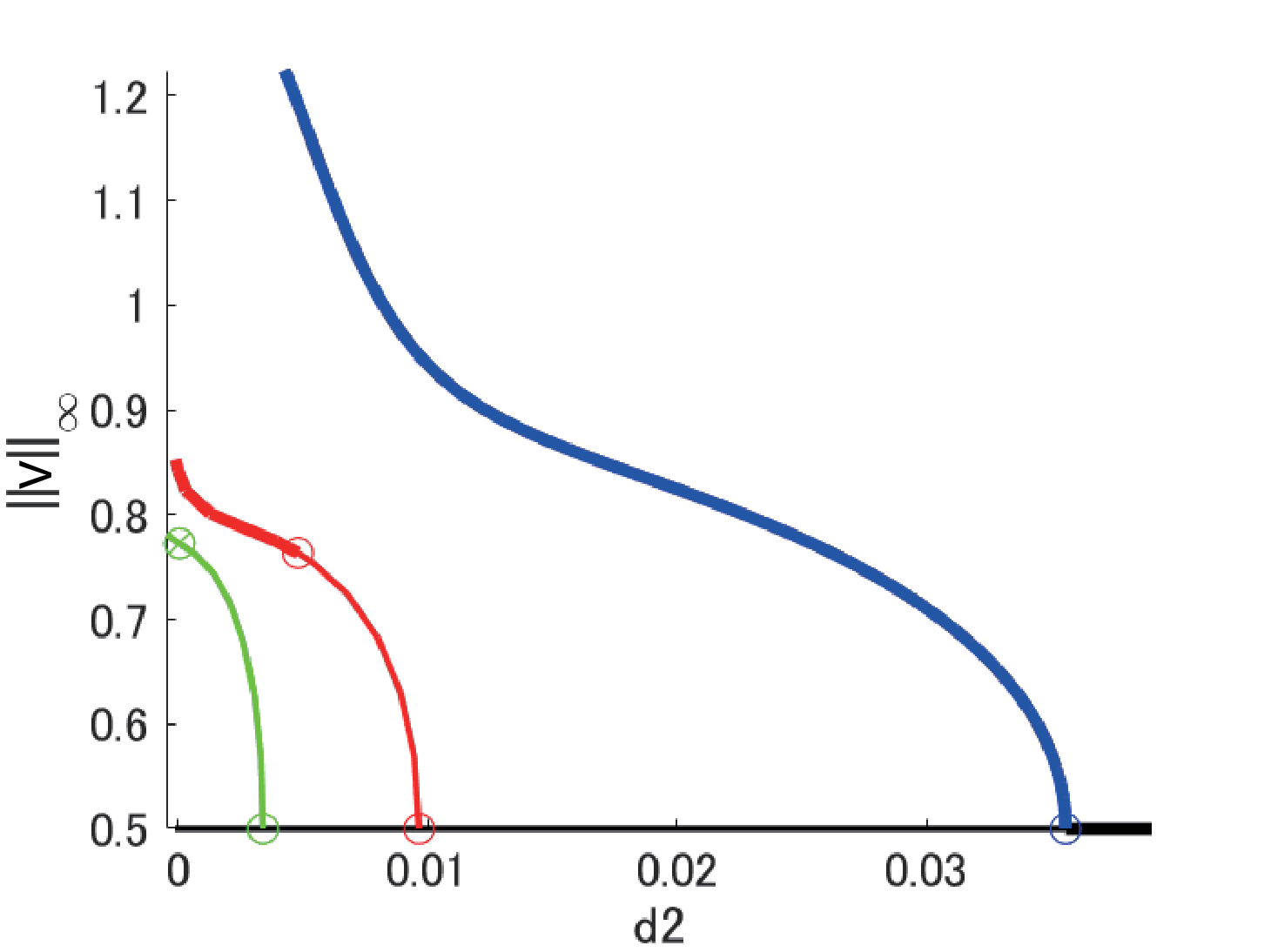}
\label{fig2a}}
\subfigure[$(\alpha, \beta )= (5,2.5)$]{
\includegraphics*[scale=.3]{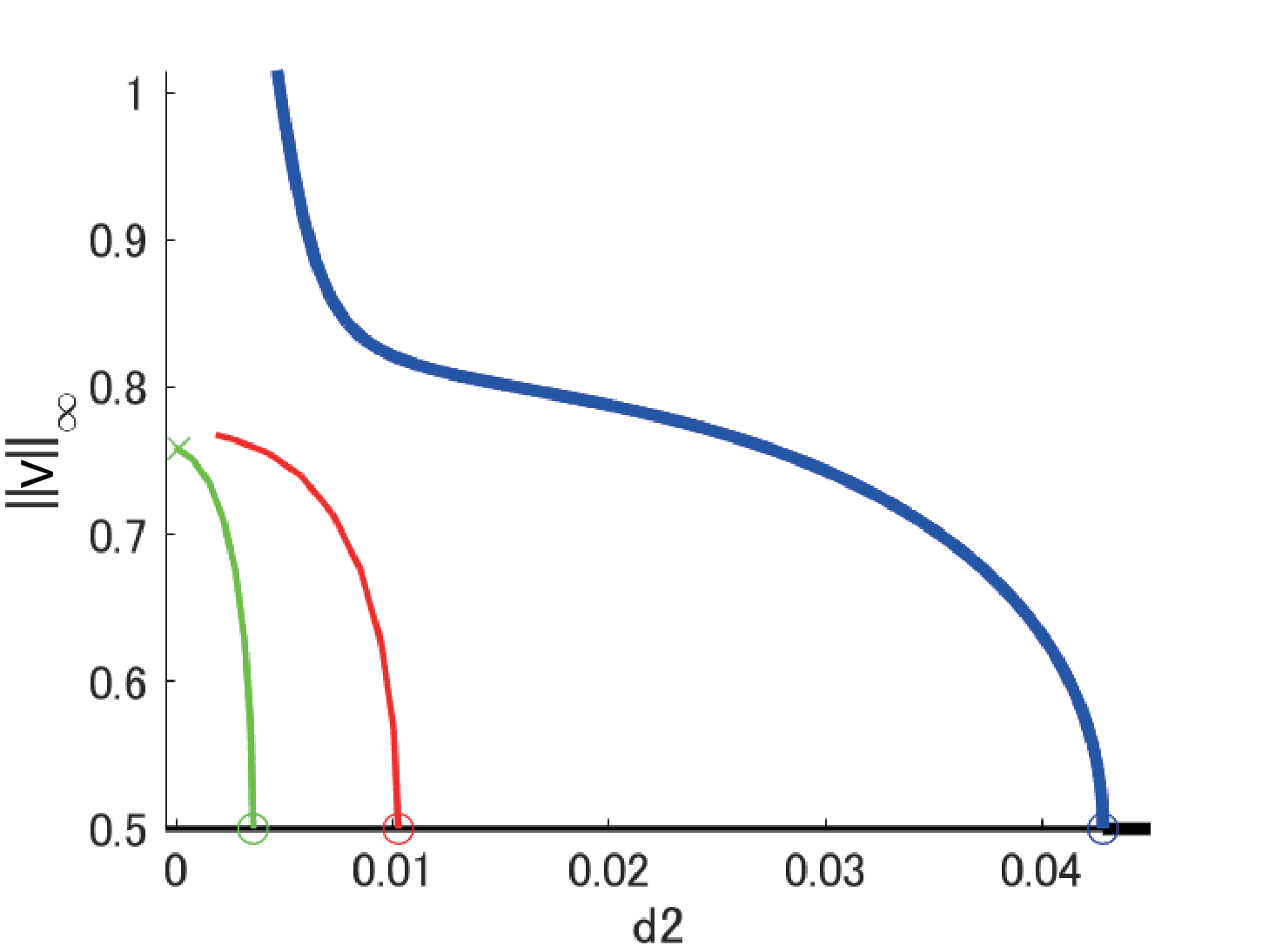}
\label{fig2b}
}
\centering
\subfigure[$(\alpha, \beta )= (10,5)$]{
\includegraphics*[scale=.3]{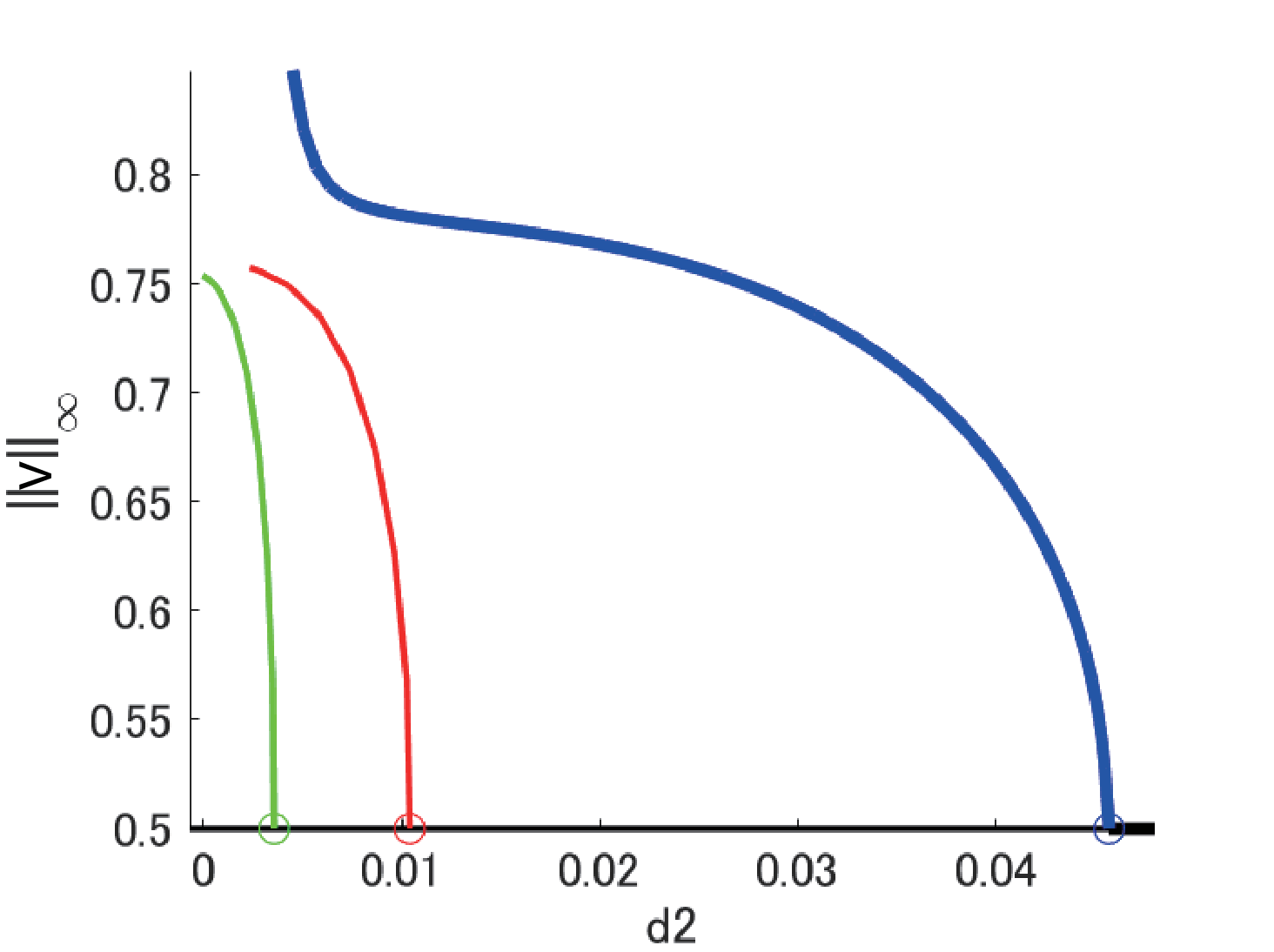}
\label{fig2c}}
\subfigure[$(\alpha, \beta )= (20,10)$]{
\includegraphics*[scale=.3]{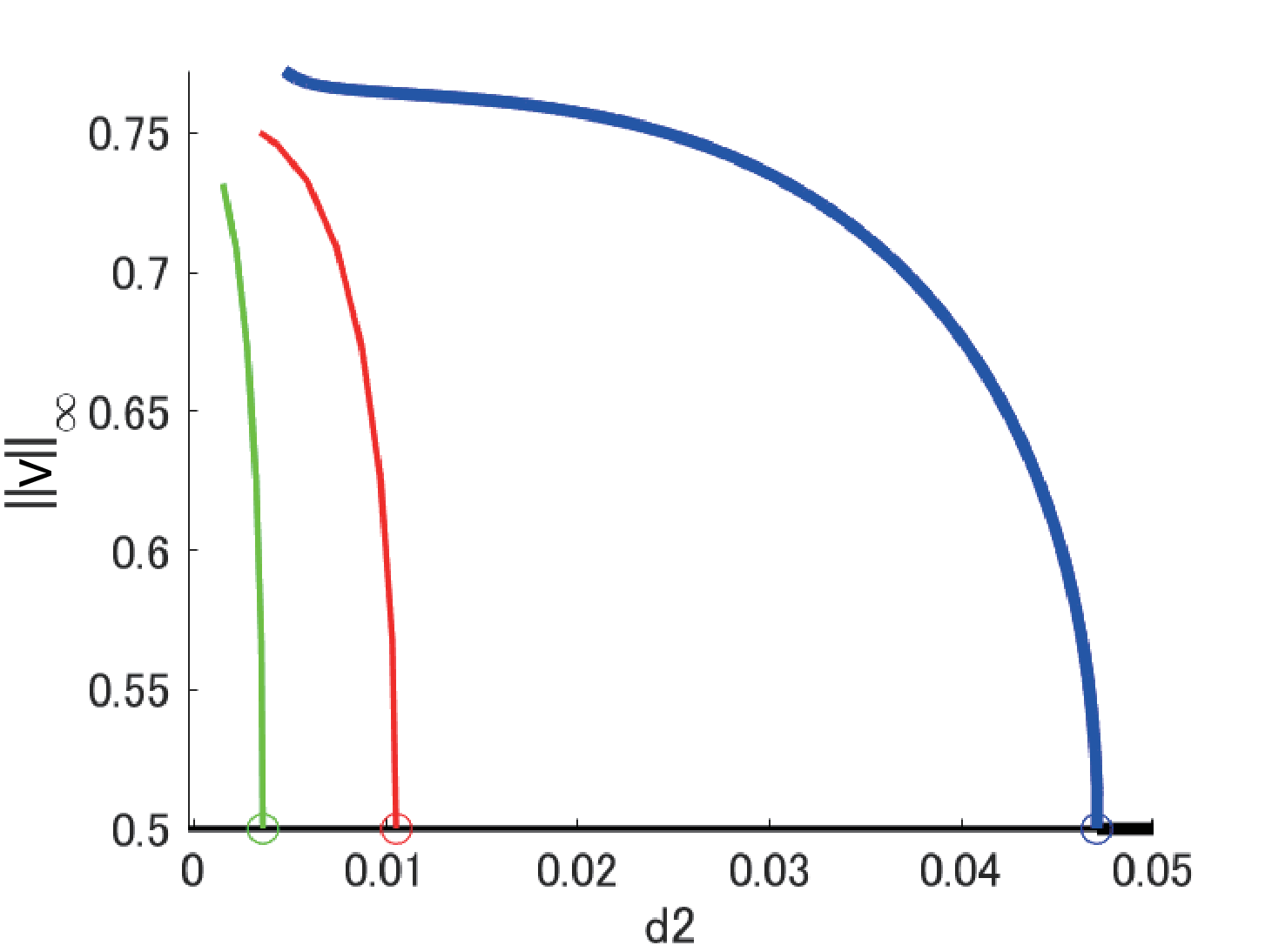}
\label{fig2d}
}
\centering
\subfigure[$(\alpha, \beta )= (50,25)$]{
\includegraphics*[scale=.3]{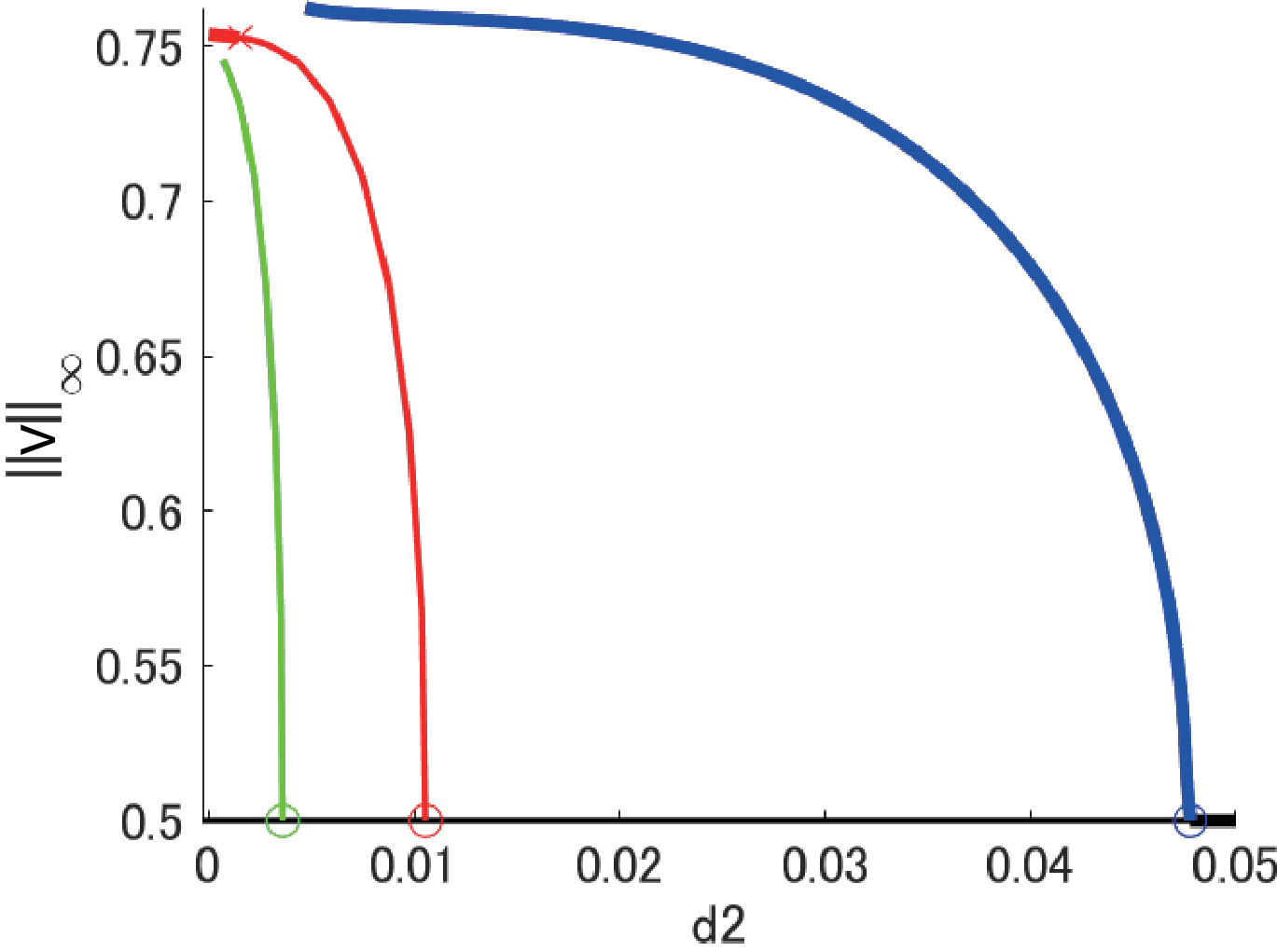}
\label{fig2e}}
\subfigure[\eqref{lim} with $\gamma =2$]{
\includegraphics*[scale=.3]{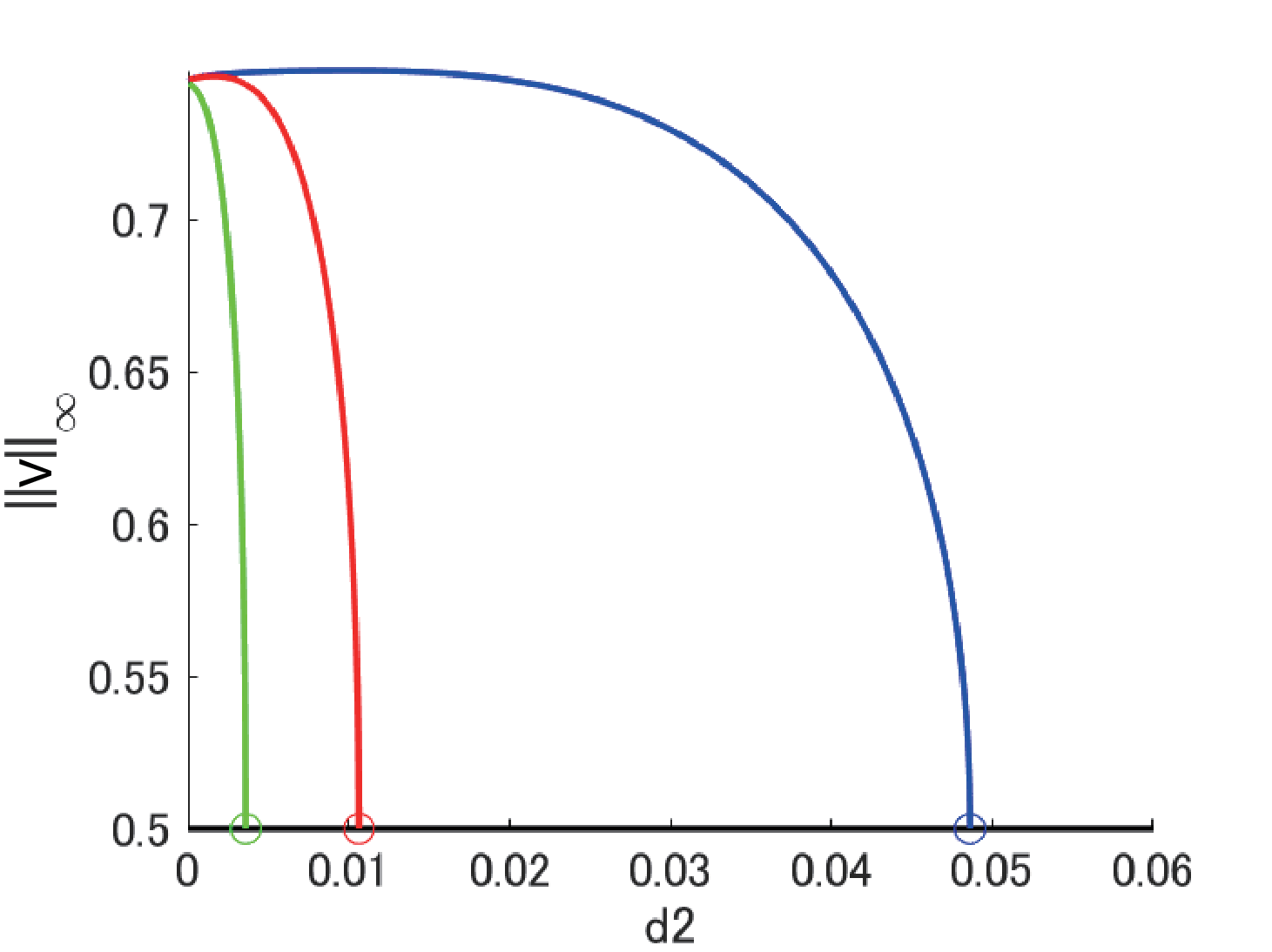}
\label{fig2f}
}
\caption{
Bifurcation branches of solutions to \eqref{st} and \eqref{lim}.}
\label{fig3}
\end{figure}

\section{Numerical bifurcation branches}
In this section, some numerical bifurcation branches of
steady-state solutions to
the one-dimensional version of \eqref{st} will be shown.
Our numerical simulations employ
the continuation software pde2path \cite{BKS, DRUW, Ue, UWR} 
based on an FEM discretization of
the stationary problem. For \eqref{st}, 
our setting of parameters in the numerical simulation is as
follows:
\begin{equation}\label{setp}
\Omega = (-0.5, 0.5), \quad
(d_1,a_1,a_2,b_1,b_2,c_1,c_2)=(0.004,1,1,4,5,2,3).
\end{equation}
and some pairs of $(\alpha, \beta )$
so that $\gamma=\alpha /\beta =2$.
In line with the theoretical analysis, 
$d_{2}$ will play an role of the bifurcation parameter.

It should be noted here that in the one-dimensional case where
$\Omega$ is an interval, 
the global bifurcation structure of 
all positive nonconstant solutions 
of the limiting system \eqref{lim} is well known.
To illustrate the structure, 
we classify the positive nonconstant solutions by node, and
introduce
$$
\mathcal{S}^{(j)}_{\infty}=\{\,(d_{2}, v)\,:\,
\mbox{$v$ satisfies \eqref{lim} and
$v'$ has exactly $j-1$ zeros in $\Omega$}\,\}.
$$
Then it is well-known that
$\mathcal{S}^{(j)}_{\infty}$ forms a pitchfork bifurcation
curve bifurcating from the positive constant solution
$v=v^{*}$ at 
$$
d_{2}=\dfrac{1}{\gamma}\biggl(
\dfrac{\xi^{*}v^{*}}{(j\pi )^{2}}-\tau^{*}d_{1}\biggr)
\,(\,=d^{(j)}_{*, \infty}\,),
$$
and moreover,
the upper and lower branches of 
$\mathcal{S}^{(j)}_{\infty}$ can be parameterized by 
$d_{2}\in (0, d^{(j)}_{*, \infty})$.
Here, for each $j\in\mathbb{N}$,
we call the subset of $\mathcal{S}^{(j)}_{\infty}$ for which $v'$ 
is monotone increasing (resp.\,decreasing) 
in a neighborhood of the left end of $\Omega$, the upper 
(resp.\,lower) branch.
In (f) of Figure 2, right (blue), center (red) and
left (green) curves numerically exhibit upper branches of
$\mathcal{S}^{(1)}_{\infty }$,
$\mathcal{S}^{(2)}_{\infty }$ and
$\mathcal{S}^{(3)}_{\infty }$
in case \eqref{setp} and $\gamma =2$.

In Figure 2, (a) to (e) numerically present the bifurcation branches 
of nonconstant solutions of \eqref{st}. 
In each of (a)-(e), right (blue), center (red) and left (green) 
bifurcation curves
correspond to global extensions of upper branches of 
$\varGamma_{1}$, $\varGamma_{2}$ and $\varGamma_{3}$
(obtained in Theorem \ref{bifthm}),
respectively.
The coefficients in the reaction terms and the linear diffusion coefficient 
$d_{1}$ are fixed to satisfy \eqref{setp}, 
while gradually increasing both $\alpha$ and $\beta$ 
by keeping their ratio $\gamma=\alpha /\beta$ equals to $2$ 
from (a) to (e). 
At this point, it can be observed that 
the bifurcation branches approach 
the bifurcation branches of the limiting system 
\eqref{lim} presented in (f).
The transition of the bifurcation branches 
from (a) to (e) supports numerically the content of (ii) of 
Theorem \ref{limthm}.

\end{document}